\documentclass{amsart}

\usepackage{amsmath,amssymb,amsthm}

\newtheorem{prop}{Proposition}[section]
\newtheorem{thm}[prop]{Theorem}

\newtheorem{cor}[prop]{Corollary}

\theoremstyle{definition}

\newtheorem*{ex}{Example}

\newtheorem{rem}[prop]{Remark}

\newtheorem*{ack}{Acknowledgement}

\def\co{\colon\thinspace}

\newcommand{\C}{\mathbb C}

\newcommand{\rmd}{\mathrm d}
\newcommand{\D}{\mathbb D}

\newcommand{\rme}{\mathrm e}

\newcommand{\rmi}{\mathrm i}

\newcommand{\MM}{\mathcal M}

\newcommand{\R}{\mathbb R}

\newcommand{\Z}{\mathbb Z}

\newcommand{\lra}{\longrightarrow}
\newcommand{\ra}{\rightarrow}

\DeclareMathOperator{\ev}{\mathrm{ev}}

\DeclareMathOperator{\Int}{\mathrm{Int}}

\DeclareMathOperator{\proj}{\mathrm{proj}}

\begin{document}

\author{Kai Zehmisch}
\address{Mathematisches Institut, Westf\"alische Wilhelms-Universit\"at M\"unster, Einsteinstr. 62,
D-48149 M\"unster, Germany}
\email{kai.zehmisch@uni-muenster.de}

\title[Analytic filling]
{Analytic filling of totally real tori}

\date{25.02.2016}

\begin{abstract}
We prove that any embedded Maslov index two
analytic disc attached to a totally real torus in 
the complex two-dimensional affine space
extends to an analytic filling
provided that the torus is contained in
a regular level set of a strictly plurisubharmonic function.
\end{abstract}

\subjclass[2010]{32E99, 58J32, 57R52, 53D35}
\thanks{KZ is partially supported by DFG grant ZE 992/1-1.}

\maketitle


\section{Introduction\label{sec:intro}}

We consider an embedded $2$-dimensional torus $T$
in the affine space $\C^2$.
We assume that $T$
is {\bf totally real} in the sense that
no complex line is tangent to the torus $T$.
We abbreviate
\[
S^1=\R/2\pi\Z
\quad\text{and}\quad
\D=\{|z|\leq1\}\subset\C\;.
\]
The totally real torus $T$ admits an
{\bf analytic filling} provided there
exists an embedding $F$ of the solid torus
$S^1\times\D$ into $\C^2$ such that
\begin{itemize}
\item [(F$_1$)]
   the boundary $S^1\times\partial\D$
   is mapped onto $T$, and
\item [(F$_2$)]
  for all $t\in S^1$ the restriction of
  $u_t:=F(t,\,.\,)$ to the interior of $\D$
  defines a holomorphic map,
  i.e., $u_t$ solves the Cauchy-Riemann equation
  \[\partial_xu_t+\rmi\partial_yu_t=0\]
  on $\Int\D$,
  where $z=x+\rmi y$.
\end{itemize}
If a thickened disc
$(-\varepsilon,\varepsilon)\times\D$
for some $\varepsilon>0$
is embedded instead the map $F$
is called a {\bf local filling}.
The aim of this note is to prove the following
extension result.

\begin{thm}
  \label{thm:main}
  If $T$ is contained in a regular
  level set of a strictly plurisubharmonic function
  on $\C^2$ then any local filling of $T$
  extends after restriction 
  to a global analytic filling of $T$.
\end{thm}

The {\bf Clifford torus} is the embedded
totally real torus
given by the product of unit circles
$\partial\D\times\partial\D$
inside $\C\times\C$
so that the solid tori $\D\times\partial\D$
and $\partial\D\times\D$ induce
an analytic filling each.
Notice that the Clifford torus
is contained in the expanded $3$-sphere
$\sqrt{2}\cdot S^3$ of all vectors having length
$\sqrt{2}$.
A small perturbation of the Clifford torus
inside $\sqrt{2}\cdot S^3$ that fixes
a neighbourhood of $\{1\}\times\partial\D$
or $\partial\D\times\{1\}$ for example
yields totally real tori that admit local fillings
but might not be foliated by circles of the Hopf fibration.
In view of Theorem \ref{thm:main}
the perturbed Clifford tori still
admit global analytic fillings.
This can be obtained by the perturbation results
of Alexander \cite{al81} and Bedford \cite{be81},
alternatively.

More generally, we consider the $2$-sphere $S^2$
in $S^3\subset\C\times\C$
obtained by intersecting
with the real hyperplane $\C\times\R$.
The intersections with the complex lines
$\C\times\{s\}$, $s\in (-1,1)$,
define an analytic filling
of $S^2$, cf.\ \cite{gz10}.
The filling collapses at the singular points
$(0,\pm1)$, at which $S^2$ is tangent to
the complex lines $\C\times\{\pm1\}$.
Attache an embedded $1$-handle to $S^2$
inside $S^3$ that is obtained
from a small tubular neighbourhood
of an embedded path that connects $(0,\pm1)$
and is everywhere transverse to the
field of complex lines $TS^3\cap\rmi TS^3$,
cf.\ \cite[Section 3.3.2]{gei08}.
The construction results in a possibly knotted
totally real torus inside $S^3$,
cf.\ \cite[Section 4.5 \& 5.3]{ho93}.
The totally real torus admits a local and, hence,
with Theorem \ref{thm:main} a global analytic filling.
A general existence result
for analytic fillings of totally real tori that are unknotted in $S^3$
is obtained by Duval--Gayet \cite{dg14}.

We remark that
the example of attaching a $1$-handle
generalizes to small perturbations
of embedded $2$-spheres in regular level sets
of strictly plurisubharmonic functions invoking
Giroux elimination lemma \cite{gi91} and local
Bishop discs \cite{bi65} or global fillings
obtained by Bedford--Gaveau \cite{bg83},
Gromov \cite{gr85},
Eliashberg \cite{el89},
Bedford--Klingenberg \cite{bk91},
Kruzhilin \cite{kr91},
and Ye \cite{ye98}.


\subsection{Totally real isotopies\label{subsec:realiso}}

By the results of Borrelli \cite{bo99,bo02}
there are infinitely many isotopy classes
of embedded totally real tori in $\C^2$.
Precisely one class admits analytic fillings,
see Proposition \ref{prop:realisot}.
In Section \ref{sec:filliso} we will prove:

\begin{cor}
 \label{cor:realiso}
 If $T$ is contained in a regular
 level set of a strictly plurisubharmonic function
 on $\C^2$ and admits a local filling,
 then $T$ is isotopic to the Clifford torus through
 totally real tori.
\end{cor}


\subsection{Non-trapped characteristics\label{subsec:nontrch}}

The complex lines tangent
to a regular level set of a strictly
plurisubharmonic function
constitute a field of real $2$-dimensional planes,
which turns out to be a contact structure,
see Section \ref{subsec:pscon}
and Section \ref{sec:gener} below.
The foliation on $T$ cut out by the contact structure
is the so-called {\bf characteristic foliation},
which is transverse to the foliation
obtained by the boundary circles of an analytic filling.
In Section \ref{subsec:nontrapping} we will show:

\begin{cor}
 In the context of Theorem \ref{thm:main}
 cylinders in $T$ cut out by boundary circles
 of holomorphic discs that belong to an analytic filling
 do not admit trapped characteristics.
\end{cor}

In view of \cite{grz14} we call a characteristic
which does not connect the two boundary components
of the cylinder to be {\bf trapped}.
The corollary says
that any boundary circle induced by an analytic filling
is a global circle of section for the characteristic
foliation of $T$.
In particular, any integrating and non-vanishing
vector field defines a Poincar\'e section map
on all boundary circles.
Hence, the characteristic foliation
is {\bf homotopically trivial},
i.e., any integrating characteristic line field
is homotopic to the kernel of a non-singular
closed $1$-form.


\section{Recollections\label{sec:recoll}}


\subsection{Pseudo-convexity\label{subsec:pscon}}

A real valued function $H$ on $\C^2$
is called {\bf strictly plurisubharmonic}
if the $2$-form $-\rmd\big(\rmd H\circ\rmi\big)$
is positive on complex lines,
i.e., \[-\rmd\big(\rmd H\circ\rmi\big)(v,\rmi v)>0\] for 
all non vanishing tangent vectors $v$
of $\C^2$.
This is equivalent to say
that $H\circ u$ is strictly subharmonic
for any holomorphic map $u\co D\ra\C^2$
defined on an open domain $D$ of $\C^2$,
cf.\ \cite[Section 3.1]{gz12}.

We assume that $T$ is contained
in the regular level set $M:=H^{-1}(0)$
and write $W=H^{-1}\big((-\infty,0]\big)$.
Then $\partial W=M$ as oriented manifolds.
The restriction of $-\rmd H\circ\rmi$
to the tangent bundle of $M$
defines a positive contact form on $M$
whose kernel $\xi$ equals $TM\cap\rmi TM$.
Therefore, $\xi$ is invariant under the multiplication
by $\rmi$.
The {\bf characteristic foliation} $T_{\xi}$
of $T$ is the intersection of $TT$ with $\xi$.
Because $T$ is a totally real torus
the intersection is transverse
so that $T_{\xi}$ is indeed a $1$-dimensional
foliation on $T$.
Choosing a co-orientation of $T$ in $M$
the complex orientation of $\xi$ orients $T_{\xi}$,
i.e., the leaves of $T_{\xi}$ the so-called
{\bf characteristic leaves} of $T$.

The strong {\bf maximum principle} of E.\ Hopf 
applied to $H\circ u$ yields
that any non-constant
holomorphic map $u\co(\D,\partial\D)\ra(\C^2,T)$
sends the interior of $\D$ into the interior of $W$
such that the restriction of $u$ to $\partial\D$
is an immersion positively transverse to $\xi$,
see \cite[Proposition 4.2]{gz10}.
Therefore, we can choose an orientation of $T$
such that the holomorphic discs
given by the local filling of $T$
intersect the leaves of $T_{\xi}$ positively.


\subsection{Factorizability\label{subsec:factor}}

One consequence of the maximum principle
is that all non-constant
holomorphic maps $u\co(\D,\partial\D)\ra(\C^2,T)$
do not have any mixed self-intersection points,
i.e., $u$ does not map an interior point of $\D$ to $u(\partial\D)$.
In view of Lazzarini's work \cite{la00, la11} this implies
that $u$ either is {\bf simple} or {\bf multiply covered}.
This means that there exist
\begin{itemize}
\item 
a holomorphic map
$\pi\co(\D,\partial\D)\ra(\D,\partial\D)$
that is continuous up to the boundary
and satisfies $\pi^{-1}(\partial\D)=\partial\D$ and
\item 
a holomorphic map
$v\co(\D,\partial\D)\ra(\C^2,T)$
with a dense set of points $z\in\D$
satisfying $T_zv\neq0$ and $v^{-1}\big(v(z)\big)=\{z\}$
\end{itemize}
such that $u=v\circ\pi$.
Being simple corresponds to $\pi$ having mapping degree $1$
as it is satisfied for $v$.


\subsection{Topological index\label{subsec:topin}}

Denote by $A$ a relative homotopy class of
continuous maps $u\co(\D,\partial\D)\ra(W,T)$.
The {\bf Maslov index} $\mu(A)$ of $A$ is defined to be
the Maslov index of the bundle pair
$\big(u^*TW,u^*TT\big)$
for any disc map $u$ representing $A$,
see \cite[Section C.3]{mcsa04}.
Notice, that the Maslov index $\mu\big([u]\big)$
does not change if the map $u$ is perturbed
through homotopies relative $T$ that even
take values outside $W\subset\C^2$.
In other words, $\mu(A)$ is uniquely determined
by the image $\partial_* A$ under the
boundary homomorphism $\partial_*$
that maps $\pi_2(\C^2,T)$ isomorphically
onto $\pi_1(T)$.

We remark that the Maslov index $\mu(A)$ is
an {\it even} integer for all classes $A$
because $T$ is orientable.
This is because the forgetful map
from the space of {\it oriented}
real planes through $0\in\C^2$
onto the space of {\it all}
real planes through $0\in\C^2$
has degree two,
cf.\ \cite[p.~52/53]{mcsa98}, \cite[Appendix]{ba94}
and \cite[p.~554]{mcsa04}.


\subsection{Intersection product\label{subsec:intprod}}

A smooth map $(\D,\partial\D)\ra(W,T)$
is called {\bf admissible} if it sends
$\Int\D$ into $\Int W$,
restricts to an immersion on $\partial\D$,
and is transvers to $M$.
Any non-constant holomorphic map
$(\D,\partial\D)\ra(\C^2,T)$ is admissible,
see Section \ref{subsec:pscon}.
For admissible maps
$u_1,u_2\co(\D,\partial\D)\ra(W,T)$
an {\bf intersection number}
$u_1\bullet u_2$ is defined
provided $u_1$ and $u_2$ intersect in
only finitely many points
and the span of the tangent spaces to
$u_1(\D)$ and $u_2(\D)$ at all
boundary intersection points
is not $3$-dimensional,
see \cite[Section 8]{gz10}.
We will say that $u_1$ and $u_2$
intersect {\bf nicely}.

To each interior intersection point
a local intersection multiplicity is assigned.
For boundary intersection points
one takes local intersection multiplicities
of extensions of $u_1(\D)$ and $u_2(\D)$
by local Schwarz reflections,
see \cite[p.~ 569-571]{gz10}.
By definition the intersection number $u_1\bullet u_2$
is the total sum of local intersection multiplicities,
where interior intersections are counted twice,
see \cite[Definition 8.8]{gz10}.
In \cite[Proposition 8.9]{gz10} it is shown
that $u_1\bullet u_2$ only depends
on the relative homotopy classes $[u_1]$ and $[u_2]$
in $\pi_2(W,T)$.
By \cite[Proposition 8.2]{gz10}
the condition on $u_1$ and $u_2$
to intersect nicely is generic.
Hence, the intersection number defines an
{\bf intersection product} $\bullet$ on $\pi_2(W,T)$,
see \cite[Remark 8.11]{gz10}.


\subsection{Non-negativity of intersections\label{subsec:posint}}

By \cite[Proposition 9.1]{gz10}
any two distinct holomorphic discs
$u_1,u_2\co(\D,\partial\D)\ra(\C^2,T)$
intersect nicely.
The intersection number $u_1\bullet u_2$
is non-negative
and equals zero if and only if 
$u_1$ and $u_2$ have disjoint images,
\cite[Theorem 9.2]{gz10}.


\subsection{Relative adjunction inequality\label{subsec:reladjineq}}

The {\bf embedding defect}
of a relative homotopy class $A$
in $\pi_2(W,T)$ is defined to be
\[
D(A)=A\bullet A-\mu(A)+2\;.
\]
For all simple holomorphic maps
$u\co(\D,\partial\D)\ra(\C^2,T)$
the embedding defect $D\big([u]\big)$
is non-negative and vanishes if and only if
$u$ is an embedding,
\cite[Theorem 9.4]{gz10}.

\begin{prop}
\label{fillimplm2}
  Let $F$ be a local filling of $T$.
  Then the Maslov index of $F(t,\,.\,)$ equals $2$
  for all $t$.
\end{prop}

\begin{proof}
 Set $u_t=F(t,\,.\,)$.
 The embedding defect $D\big([u_0]\big)$ vanishes
 because $u_0$ is an embedding.
 The intersection product $[u_0]\bullet[u_0]$
 is equal to the intersection number $u_0\bullet u_t$
 for $t>0$ small,
 which is zero as $u_0(\D)$ and $u_t(\D)$ are disjoint.
 Therefore, $\mu\big([u_0]\big)=2$.
\end{proof}

\begin{ex}
 Consider the holomorphic embedding
 $u(z)=(z,z)$, $z\in\D$, that takes boundary values
 on the Clifford torus $\partial\D\times\partial\D$ inside $\sqrt{2}\cdot S^3$.
 As $u$ has Maslov index $4$ there is no
 filling of the Clifford torus that extends $u$,
 which of course can be verified directly.
\end{ex}

\begin{rem}
 In \cite{kr91} Kruzhilin showed
 that any totally real torus
 that is contained in a regular level set
 of a strictly plurisubharmonic function
 admits a family
 of Maslov index $2$ holomorphic discs that
 are attached to the torus passing though
 all its points.
 Therefore, the total obstruction to extend to an
 analytic filling lies in the vanishing of the
 self-intersection number of the represented
 relative homology class.
\end{rem}


\subsection{Automatic transversallity\label{subsec:autotrans}}

We formulate the converse of
Proposition \ref{fillimplm2}.
For that consider a 
holomorphic embedding
$u\co(\D,\partial\D)\ra(\C^2,T)$.
We call the collection of
three pairwise disjoint local paths in $T$
that intersect $u(\partial\D)$ transversally
in a single point resp.\ the
{\bf transverse constraints}.

\begin{prop}
  \label{hofimpl}
 Let $u\co(\D,\partial\D)\ra(\C^2,T)$
 be a holomorphic embedding
 of Maslov index $2$.
 Then there exists a local filling $F$ of $T$
 that extends $u=F(0,\,.\,)$ such that
 the curves $t\mapsto F(t,\rmi^k)$,
 $k=0,1,2$, locally
 parametrize the transverse constraints.
 Moreover, the filling is unique up to
 re-parametrizations in the time variable $t$
 and shrinking the time interval.
\end{prop}

\begin{proof}
 This is worked out in \cite[Section 2]{ho93}
 and \cite[Section 3.2]{ho99}.
 The necessary modifications
 in view of the transverse constraints,
 can be achieved similarly to
 \cite[Lemma 7.5 and Proposition 7.6]{gz10}
 by a choice of a Riemannian metric that turns $T$
 and the three transverse constraints
 into totally geodesic submanifolds.
\end{proof}


\section{A moduli space\label{sec:amodsp}}

We assume the situation of Theorem \ref{thm:main}
and denote the local filling of $T$ by $F$.


\subsection{Definition\label{subsec:definition}}

Set $u_0=F(0,\,.\,)$.
Provide the space of all holomorphic maps
$u\co(\D,\partial\D)\ra(\C^2,T)$,
which by Section \ref{subsec:pscon}
take values in $W$,
with the $C^{\infty}$-topology.
The subspace consisting
of all $u$ that are homologous to $u_0$
in $W$ relative $T$ is denoted
by $\widetilde{\MM}$.
Notice, that all holomorphic discs
$u\in\widetilde{\MM}$ are simple.
To see this factor $u=v\circ\pi$ as
described in Section \ref{subsec:factor}.
Because $\mu\big([v]\big)$ is even and
Proposition \ref{fillimplm2}
the degree of the holomorphic map
$\pi$ must be one.
In particular,
the group $G$ of conformal
automorphisms of $(\D,\rmi)$
acts without fixed points
via re-parametrizations.
The {\bf moduli space} $\MM$ is defined to be
the quotient $\widetilde{\MM}/G$.

Because $D\big([u_0]\big)$ vanishes
all holomorphic maps in $\widetilde{\MM}$
are embeddings with Maslov index $2$,
cf.\ Proposition \ref{fillimplm2}.
By Section \ref{subsec:posint}
the images of holomorphic maps in $\widetilde{\MM}$
that have distinct images,
i.e., represent distinct classes in $\MM$,
are in fact disjoint
because of $[u_0]\bullet[u_0]=0$.
With the arguments from
Proposition \ref{hofimpl}
$\widetilde{\MM}$ is a $4$-dimensional,
and hence $\MM$ a $1$-dimensional,
smooth manifold.


\subsection{Geometric bounds\label{subsec:geombounds}}

Denote by $B$ a ball in $\C^2$ that contains the torus $T$.
Because $B$ is a sub-level set of the strictly plurisubharmonic
function $\mathbf{z}\mapsto\frac14|\mathbf{z}|^2$
all holomorphic discs $u(\D)$ with $u(\partial\D)\subset T$
are contained in $B$.
This again follows from the maximum principle.


\subsection{Cutting the torus\label{subsec:cuttor}}

The image of $u_0|_{\partial\D}$ is an oriented knot in $T$,
which is transverse to the characteristic leaves of $T$,
see Section \ref{subsec:pscon}.
Therefore, $u_0(\partial\D)$ can not bound a disc inside $T$.
In view of \cite[p.~ 25]{ro76}
the complement $T\setminus u_0(\partial\D)$
is diffeomorphic to $S^1\times (0,1)$.
We may assume that the boundary circles
of the holomorphic discs that belong to the local filling $F$
coincide with the slices $S^1\times\{t\}$
for $t\in(0,\varepsilon)\cup(1-\varepsilon,1)$
so that we can add two copies of $u_0(\partial\D)$
to the cylinder to get $S^1\times [0,1]$.
According to the orientation convention
in Section \ref{subsec:pscon} the leaves of $T_{\xi}$
point inwards $S^1\times [0,1]$ along $S^1\times\{0\}$.


\subsection{Energy bounds\label{subsec:enbounds}}

Denote by $\omega$ the symplectic form
$\rmd\mathbf{x}\wedge\rmd\mathbf{y}$
of $\C^2\equiv\R^4$.
The energy $E(u)$
of a holomorphic disc $u$
is defined by $\int_{\D}u^*\omega$.
Because $\rmi$ and $\omega$ are compatible
$E(u)$ is equal to the Dirichlet energy of $u$,
see \cite[Section 2.2]{mcsa04}.
Both descriptions imply
that the energy is invariant under conformal
re-parametrizations.

\begin{prop}
 \label{energest}
 The energy function $[u]\mapsto E(u)$
 on $\MM$ is bounded from above.
\end{prop}

\begin{proof}
 We consider a holomorphic map $u\in\widetilde{\MM}$
 that is not a conformal re-parametri-zation of $u_0$.
 Then $u(\partial\D)$ divides 
 $T\setminus u_0(\partial\D)$,
 which is diffeomorphic to $S^1\times(0,1)$,
 into two non-empty cylindrical components,
 because the knots $u_0(\partial\D)$ and $u(\partial\D)$
 are homotopic in $S^1\times[0,1]$.
 We denote the cylinder for which $u(\partial\D)$
 is an oriented boundary component by $C$.
 Therefore, by Stokes theorem
 \[
 \int_{C}\omega=
 \int_{u(\partial\D)}\!\!\lambda-
 \int_{u_0(\partial\D)}\!\!\lambda
 =E(u)-E(u_0)
 \]
 writing $\lambda$ instead of $\mathbf{x}\rmd\mathbf{y}$.
 Choose an area form $\sigma$ on $T$
 so that $\omega=f\sigma$ for a smooth function
 $f$ on the $2$-torus $T$.
 Therefore,
 \[
 \int_{C}\omega\leq
 \max_{C}|f|\cdot\sigma(C)\;,
 \]
 where we denote the total area of a subset
 $U\subset T$ by $\sigma(U)$.
 Combining both expressions we obtain
 \[
 E(u)\leq E(u_0)+
 \max_{C}|f|\cdot\sigma(C)\;.
 \]
 Repeating the argument with $C$ replaced by $T\setminus C$
 we eventually obtain
 \[
 E(u)\leq E(u_0)+
 \tfrac12\max_{T}|f|\cdot\sigma(T)
 \]
 because the minimum of two real numbers
 is smaller than the arithmetic mean.
\end{proof}


\subsection{Compactness\label{subsec:cptness}}

Any sequence in $\widetilde{\MM}$
has a subsequence that Gromov converges 
to a stable holomorphic disc $\mathbf{u}$,
which represents the relative homotopy class $[u_0]$,
see \cite[Theorem 1.1]{fz15}.
By Liouville's theorem $\mathbf{u}$
has no spherical components.

\begin{prop}
 \label{prop:cptness}
 $\MM$ is compact.
\end{prop}

\begin{proof}
 We consider a sequence in $\widetilde{\MM}$.
 We can assume that the boundary circles
 of the holomorphic discs stay in the complement
 of $F\big((-\varepsilon/2,\varepsilon/2)\times\partial\D\big)$.
 Denote by $u^1,\ldots,u^N$ the components of
 a limiting stable holomorphic disc
 of a Gromov converging subsequence.
 We conclude that the circles $u^j(\partial\D)$
 are contained in the complement of $u_0(\partial\D)$,
 i.e., in $S^1\times(0,1)$.
 We will show that $N=1$
 so that the chosen subsequence descends to
 a converging sequence in $\MM$.
 
 As described in Section \ref{subsec:factor}
 each of the holomorphic discs $u^j$
 factors through a simple holomorphic disc $v^j$
 via a branched covering map of degree $m_j\geq1$.
 Therefore, $[u_0]$ equals
 \[
 [u_0]=\sum_{j=1}^Nm_j[v^j]
 \]
 in $\pi_2(W,T)$.
 We can assume that the images of the $v^j$'s
 are pairwise distinct.
 If not we combine any pair of classes
 that represent $v^j$'s with common image
 to a single class weighted with the sum of
 the multiplicities.
 This procedure shrinks $N$ and enlarges
 the $m_j$'s.
 
 We will utilize the argument on \cite[p.~ 549/50]{gz13}
 and assume by contradiction that $N\geq2$.
 Because $[u_0]\bullet[u_0]=0$
 we get with Section \ref{subsec:posint}
 that $[u_0]\bullet[v^j]=0$ for all $j$.
 Substituting the above expression for $[u_0]$
 once more we get for all $j$
 \[
 0=\sum_{k=1}^Nm_k[v^k]\bullet[v^j]
 \]
 or equivalently
 \[
 -m_j[v^j]\bullet[v^j]=\sum_{k\neq j}m_k[v^k]\bullet[v^j]\;.
 \]
 At least two of the $v^j$'s,
 which originate from the bubble tree of the limiting
 stable holomorphic disc,
 must intersect.
 Section \ref{subsec:posint} yields
 that the right hand side is positive.
 Therefore, $[v^j]\bullet[v^j]\leq-1$ for all $j$.
 
 Because $\mu\big([u_0]\big)=2$
 we get furthermore
 \[
 2=\sum_{j=1}^Nm_j\mu\big([v^j]\big)\;.
 \]
 Moreover,
 by Section \ref{subsec:reladjineq}
 the embedding defect
 \[0\leq D\big([v^j]\big)=[v^j]\bullet[v^j]-\mu\big([v^j]\big)+2\]
 of the $[v^j]$'s is non-negative.
 Combining both yields
 \[
 2\leq\sum_{j=1}^Nm_j\Big([v^j]\bullet[v^j]+2\Big)\;.
 \]
 We conclude that at least one of the $v^j$'s,
 $v^1$ say, has self-intersection number
 equal to $[v^1]\bullet[v^1]=-1$.
 
 We can approximate
 $v^1$ by a smooth admissible map
 $w\co(\D,\partial\D)\ra(W,T)$
 that represents $[v^1]$ and intersects $v^1$
 transversely in a finite number of points,
 see \cite[Remark 8.6]{gz10}
 and Section \ref{subsec:intprod}.
 Hence, $v^1\bullet w=-1$.
 The contributions from interior intersections
 to the intersection number $v^1\bullet w$
 is even because they are counted twice.
 Therefore,
 $v^1(\partial\D)$ and $w(\partial\D)$
 intersect in an {\it odd} number of points,
 which are transverse intersections by construction.
 This implies that the ordinary intersection product
 \[
 \big[v^1|_{\partial\D}\big]\cdot\big[v^1|_{\partial\D}\big]
 =
 \big[v^1|_{\partial\D}\big]\cdot\big[w|_{\partial\D}\big]
 \neq0
 \]
 does not vanish
 on the first homology of $T\setminus u_0(\partial\D)$.
 In other words
 $\big[v^1|_{\partial\D}\big]\neq0$,
 so that the classes of $v^1|_{\partial\D}$
 and $u_0|_{\partial\D}$ are non-trivially co-linear
 in
 $H_1\big(T\setminus u_0(\partial\D)\big)$,
 the latter being identified with
 $\Z\big[u_0|_{\partial\D}\big]$.
 We infer that
 \[
 \big[u_0|_{\partial\D}\big]\cdot\big[u_0|_{\partial\D}\big]
 \neq0\;.
 \]
 This is a contradiction as $u_0(\partial\D)$
 and $F(\varepsilon/2,\partial\D)$ are disjoint.
 Hence, $N=1$.
\end{proof}


\section{Extensions of local fillings\label{sec:exoflocfill}}


\subsection{Holomorphic discs with one boundary
 marked point\label{subsec:markedpoit}}

Consider the quotient space
\[
\MM^1=\widetilde{\MM}\times_G\partial\D
\]
by the action
\[
g*(u,z)=\big(u\circ g, g^{-1}(z)\big)\;.
\]
By Section \ref{subsec:definition}
and Proposition \ref{prop:cptness}
the moduli space
$\MM^1$ is a closed surface.
Charts can be obtained via local fillings
as described in Proposition \ref{hofimpl}.
For any $u\in\widetilde{\MM}$ and
transverse constraints $c_0,c_1,c_2$
that intersect $u(\partial\D)$
in $u(1),u(\rmi),u(-1)$, resp.,
there exists a local filling
$F\co(-\varepsilon,\varepsilon)\times\D\ra\C^2$
that extends $u=F(0,\,.\,)$.
The local filling is 
uniquely determined up to
re-parametrizations in time.
The map
$(-\varepsilon,\varepsilon)\times\partial\D\ra\MM^1$
given by
$(t,z)\mapsto\big[F(t,\,.\,),z\big]$
is a local parametrization near $(u,z)$.
In order to obtain coordinate changes
consider local fillings $F_1$ and $F_2$
that extend $u$ and $u\circ h$,
$h\in G$, resp.
By shrinking the time intervals
we can assume that the images coincide.
Therefore, we obtain a diffeomorphism
\[
F_2^{-1}\circ F_1=(f,g)\co
(-\varepsilon_1,\varepsilon_1)\times\D
\lra
(-\varepsilon_2,\varepsilon'_2)\times\D\;,
\]
where $f\co(-\varepsilon_1,\varepsilon_1)\ra
(-\varepsilon_2,\varepsilon'_2)$ is a smooth
strictly increasing function and
$g=g_t$, $t\in(-\varepsilon_1,\varepsilon_1)$,
a smooth $1$-parameter family of
conformal automorphisms in $G$
such that $f(0)=0$ and $g_0=h^{-1}$,
see \cite[Theorem 18]{ho93} and
\cite[Proposition 3.12]{ho99}.
The desired coordinate change
$(-\varepsilon_1,\varepsilon_1)\times\partial\D
\ra
(-\varepsilon_2,\varepsilon'_2)\times\partial\D$
according to the parametrizations
$F_1$ and $F_2$
spells out
as $(t,z)\mapsto\big(f(t),g_t(z)\big)$.
In particular, the evaluation map
\[
\ev\co\MM^1\lra T\;,\qquad
[u,z]\longmapsto u(z),
\]
equals $(t,z)\mapsto F(t,z)$
in the chart obtained by $F$.
Because $\ev$
is injective by Section \ref{subsec:posint}
and Section \ref{subsec:reladjineq}
the evaluation map $\ev$
is a diffeomorphism.


\subsection{No trapped characteristics on
 the cut off torus\label{subsec:nontrapping}}

We identify $T\setminus u_0(\partial\D)$ with
$S^1\times (0,1)$ according to the conventions
in Section \ref{subsec:cuttor}.

\begin{prop}
\label{through}
 Each characteristic leaf of $S^1\times (0,1)$
 connects $S^1\times\{0\}$
 with $S^1\times\{1\}$.
\end{prop}

\begin{proof}
  Choose a characteristic leaf
  $\ell$ on $S^1\times (0,1)$
  that intersects $S^1\times\{0\}$.
  For each point $p\in\ell$
  there exists a local filling $F_p$
  such that $F_p(0,\partial\D)$
  intersects $p$,
  see Section \ref{subsec:markedpoit}.
  Denote by
  \[
  U=\bigcup_{p\in\ell}
  F_p\Big(
  (-\varepsilon_p,\varepsilon_p)
  \times\partial\D
  \Big)
  \]
  the union of the images of all local fillings
  $F_p$, $p\in\ell$, in $S^1\times (0,1)$.
  Denote by $s_0$ the supremum of
  \[
  \proj_2\Big(U\cap\big(\{0\}\times(0,1)\big)\Big)
  \]
  in the unit interval $(0,1)$.
  Then there exists a local filling $F_0$
  such that $F_0(0,\partial\D)$
  intersects the point $(0,s_0)$ in $S^1\times (0,1)$.
  The image of $F_0$ in
  $S^1\times (0,1)$ overlaps with $U$.
  Because $\ell$ intersects boundary circles
  of holomorphic discs transversally
  each boundary circle of each
  holomorphic disc $F_0(t,\,.\,)$
  intersects $\ell$ non-trivially.
  Hence, $s_0=1$ and $F_0(0,\partial\D)$
  coincides with $S^1\times\{1\}$
  so that
  $\ell$ connects $S^1\times\{0\}$
  with $S^1\times\{1\}$.
\end{proof}


\subsection{Gluing local fillings\label{subsec:guingloc}}

We consider a local filling
$F\co(-\varepsilon,\varepsilon)\times\D\ra\C^2$
of $T$ and set $u_0=F(0,\,.\,)$.
In view of Proposition \ref{through}
we choose three disjoint smooth knots
$K_0,K_1,K_2$ in $T$
such that each coincides
with a connected characteristic
leaf on $S^1\times[\varepsilon,1-\varepsilon]$
and is transverse to the boundary circles
of $F(t,\,.\,)$,
whose images correspond to the slices in
$S^1\times\big((0,\varepsilon)\cup(1-\varepsilon,1)\big)$.
Therefore, each holomorphic disc
$u\in\widetilde{\MM}$ intersects each knot 
$K_k$ transversely.
We can assume that $u_0(\rmi^k)\in K_k$
for $k=0,1,2$.

Let $\MM_{1,\rmi,-1}$ be the moduli space
of all $u\in\widetilde{\MM}$
such that 
in $u(1)\in K_0,u(\rmi)\in K_1,u(-1)\in K_2$.
Requiring the three marked points
to lie on the respective knots,
which play the role of {\it global}
transverse constraints,
is the same as to build
the abstract quotient $\MM$.
A variant of the considerations
in Section \ref{subsec:markedpoit}
that ignores the marked point
shows that
\[
\MM_{1,\rmi,-1}\lra\MM\;,\qquad
u\longmapsto [u]\;,
\]
is a diffeomorphism.
Therefore, $\MM_{1,\rmi,-1}$
is a circle,
which of course can be seen directly with
Proposition \ref{hofimpl} and
\cite[Proposition 7.1]{gz10}.

The evaluation map
\[
\ev\co\MM_{1,\rmi,-1}\times\partial\D
\lra T\;,\qquad
(u,z)\longmapsto u(z)\;,
\]
is a diffeomorphism,
as it factors though the diffeomorphism
\[
\MM_{1,\rmi,-1}\times\partial\D\lra\MM^1\;,\qquad
(u,z)\longmapsto [u,z]\;,
\]
and the evaluation map we considered earlier,
see Section \ref{subsec:markedpoit}.
By construction $\ev_1=\ev(\,.\,,1)$
maps $\MM_{1,\rmi,-1}$ diffeomorphically
onto the knot $K_0$.
Therefore,
choosing a regular parametrization
$c\co S^1\ra K_0$
the map
\[
S^1\times\D\lra\C^2\;,\qquad
(t,z)\longmapsto\Big(\ev_1^{-1}\!\big(c(t)\big)\Big)(z)\;,
\]
turns out to be an analytic filling of $T$
cf.\ \cite[Proposition 5.2]{gz10}.

\begin{proof}[{\bf Proof of Theorem \ref{thm:main}}]
 Denote by $F_1$ the global analytic filling just obtained.
 The given local filling is denoted by $F_2$.
 Then $F_2^{-1}\circ F_1$ has the form $(f,g)$
 as indicated in Section \ref{subsec:markedpoit}.
 This time $f$ sends $(a,b)$ to $(-\varepsilon,\varepsilon)$,
 where we can assume that $a\in(-\pi,0)$ and $b\in(0,\pi)$.
 The path $g=g_t$ in $G$ is parametrized by $t\in(a,b)$.
 Replace $f$ by a smooth strictly increasing function
 that is the identity near $0$ and coincides with
 the old near $\pm\varepsilon$.
 Similarly, replace the path $g_t$
 by a path in $G$ that is constantly equal to the identity
 near $0$ and coincides with the old path
 near $\pm\varepsilon$.
 This is possible because the group $G$
 of M\"obius transformations
 on $\D$ is diffeomorhic to $S^1\times\R$
 and, hence, connected.
 The analytic filling that equals
 $F_2\circ (f,g)$ on $(a,b)\times\D$
 and $F_1$ elsewhere on $S^1\setminus(a,b)\times\D$
 is an extension after restriction of $F_2$.
\end{proof}


\section{Generalizations\label{sec:gener}}

Theorem \ref{thm:main} remains valid if we
replace $\C^2$ by any {\bf Stein surface},
which is a complex $2$-dimensional complex
manifold that admits a proper holomorphic
embedding into a complex affine space,
cf.\ \cite[p.~283/4]{gei08}.
In fact the integrability of the almost complex
structure is not used in the proof.
The theorem therefore can be phrased in
the following form:

Let $(M,\xi)$ be a closed oriented contact $3$-manifold
with a {\bf positive contact form} $\alpha$,
i.e., $\alpha\wedge\rmd\alpha>0$.
Let $(W,\omega)$ be a {\bf weak filling} of $(M,\xi)$,
i.e., a compact symplectic manifold,
which is oriented by $\omega^2$,
such that $\partial W=M$ as oriented manifolds
and the restriction of $\omega$ to $\xi$
is positive.
Assume that the symplectic manifold admits an almost
complex structure $J$ that
turns $T\subset M$ into a totally totally real torus,
leaves $\xi$ and its symplectic orthogonal invariant,
and is {\bf tamed} by $\omega$,
i.e., $\omega$ is positive on $J$-complex lines.
In particular, the boundary $M$
is $J$-{\bf convex}, i.e., the complex tangencies
$\xi$ define a positive contact structure,
see \cite[p.~538]{gz13} and \cite[Remark 4.3]{gz10}.

Assume that $T$ admits a local filling $F$.
Then $F$ extends to a global filling of $T$
after restriction, either if $H_2W$ has no
spherical classes $A$ with $\omega(A)>0$
and $A\cdot A=-1$, or $J$ is perturbed away
from $F\big([-\varepsilon/2,\varepsilon/2]\times\D\big)$
to be generic for precisely those classes
and $(W,\omega)$ is {\bf minimal},
i.e., does not admit any embedded symplectic
sphere of self-intersection equal to $-1$.
Indeed, if $J$ turns all simple stable holomorphic discs
that have an {\bf exceptional sphere}
component as described
into a regular moduli space problem
no further assumption on $W$ are necessary.

The proof is a variation of the one we gave.
In order to get uniform energy bounds
replace the primitive $\lambda$ in Proposition \ref{energest}
by a $3$-chain that has boundary
$u(\D)+C-u_0(\D)$,
which exists by assumption.
Moreover, combine the intersection argument
from Proposition \ref{prop:cptness}
with the one from
\cite[p.~ 549/50]{gz13}, and from
\cite[p.~276]{gz12} in the generic case described at last,
in order to prove compactness.


\section{From filling to isotopy\label{sec:filliso}}

We consider a totally real torus $T\subset\C^2$
which is provided with an analytic filling
$F\co S^1\times\D\ra\C^2$.
We write $u_t$ instead of $F(t,\,.\,)$
for all $t\in S^1$
and parametrize the Clifford torus by
\[
S^1\times\partial\D\lra\C^2\;,\qquad
\big(t,\rme^{\rmi\theta}\big)\longmapsto
\big(\rme^{\rmi t},\rme^{\rmi\theta}\big)\;.
\]
Corollary \ref{cor:realiso} will follow from the following:

\begin{prop}
\label{prop:realisot}
 There exists an isotopy $\Phi$
 of embeddings $S^1\times\partial\D\ra\C^2$
 that connects $\Phi_0=F$ with the chosen
 parametrization $\Phi_1$ of the Clifford torus
 such that the image tori $\Phi_s(S^1\times\partial\D)$
 are totally real for all $s\in[0,1]$.
\end{prop}

\begin{proof}
 By the unknotting theorem
 \cite[Corollary 7.2]{ad93}
 the embeddings $F(\,.\,,0)$
 and $t\mapsto\big(\rme^{\rmi t},0\big)$
 of $S^1$ into $\C^2$ are isotopic.
 We denote the isotopy by
 \[
 \varphi\co [0,1]\times S^1\lra\C^2\;,\qquad
 (s,t)\longmapsto\varphi_s(t)\;.
 \]
 The isotopy defines a trivial
 complex plane bundle $E=\varphi^*T\C^2$,
 which has a preferred non-vanishing
 section $X$ defined by
 $X_{(s,t)}=T_t\varphi_s(\partial_t)$.
 The complex line bundle generated
 by $X$ is denoted by $E_1$.
 Let $E_2$ be a complementary complex
 subbundle of $E_1$ in $E$
 such that the fibre of $E_2$ over
 $(0,t)$ equals the tangent plane
 of the disc $u_t(\D)$ at $u_t(0)$
 and the fibre of $E_2$ over
 $(1,t)$ is equal to
 $\{\rme^{\rmi t}\}\times\C$.
 Choose a complex trivialization
 of $E$ that preserves the splitting
 $E_1\oplus E_2$.
 Together with the normal exponential map
 of the Euclidean metric we obtain
 an isotopy $\Phi$ of embeddings
 \[
 \big(S^1\times (-\varepsilon,\varepsilon)\big)
 \times D^2_{\varepsilon}\lra\C^2
 \]
 for some small $\varepsilon>0$.
 The isotopy $\Phi$ restricts to $\varphi$ on
 $[0,1]\times S^1\times\{0\}\times\{0\}$
 such that
 $\Phi_1(t,r;z)=\big(\rme^{\rmi t},z\big)$
 and that $\Phi_s^*\rmi$ leaves the splitting
 $\big(S^1\times (-\varepsilon,\varepsilon)\big)
 \times D^2_{\varepsilon}$
 invariant.

 The desired isotopy will be a combination
 of isotopies.
 In a first, we shrink $T$ via
 \[
 \big(t,\rme^{\rmi\theta}\big)
 \longmapsto
 F\big(t,r\rme^{\rmi\theta}\big)
 \]
 into the image of $\Phi_0$,
 which is a neighbourhood of $F(S^1,0)$,
 by making $r$ small.
 The neighbourhood contains a second family
 of embedded totally real tori that are
 parametrized by
 $\Phi_0\big(t,0;r\rme^{\rmi\theta}\big)$.
 The tangent spaces to the tori
 for $t$ and $\theta$ fixed converge
 in both families to the real plane spanned
 by the vectors $X_{(0,t)}$ and
 $T_{(t,0;0)}\Phi_0\big(0,\rmi\rme^{\rmi\theta}\big)$
 as $r$ tends to zero.
 Hence, for $r$ sufficiently small the angles
 between the respective tangent
 planes is less than $\pi/2$
 so that the convex combination of
 $F\big(t,r\rme^{\rmi\theta}\big)$
 and
 $\Phi_0\big(t,0;r\rme^{\rmi\theta}\big)$
 induces an isotopy of embeddings.
 Shrinking $r$ again if necessary
 the isotopy will be through totally real tori
 because the space of all real planes
 in the Gra{\ss}mannian of all $2$-dimensional
 subspaces of $\C^2\equiv\R^4$ is open.
 The third isotopy is given by
 \[
 \big(t,\rme^{\rmi\theta}\big)
 \longmapsto
 \Phi_s\big(t,0;r\rme^{\rmi\theta}\big)
 \]
 sending $s$ to $1$
 and the last brings
 $\partial D^2_1\times\partial D^2_r$
 to the Clifford torus
 by radially expanding
 the second circle factor.
\end{proof}


\begin{ack}
  I would like to thank Matthias Schwarz
  from whom I learned the technique of
  {\it filling with holomorphic discs}.
  I would like to thank Peter Albers, Gabriele Benedetti, and
  Hansj\"org Geiges for their interest in this work.
\end{ack}


\end{document}